\documentclass[12pt,reqno]{article}

\usepackage[usenames]{color}
\usepackage{amssymb}
\usepackage{amsmath}
\usepackage{amsthm}
\usepackage{amsfonts}
\usepackage{amscd}
\usepackage{graphicx}

\usepackage[colorlinks=true,
linkcolor=webgreen,
filecolor=webbrown,
citecolor=webgreen]{hyperref}

\definecolor{webgreen}{rgb}{0,.5,0}
\definecolor{webbrown}{rgb}{.6,0,0}

\usepackage{color}
\usepackage{fullpage}
\usepackage{float}
\usepackage{mathrsfs}

\usepackage{graphics}
\usepackage{latexsym}
\usepackage{epsf}

\newtheorem{ex}{Example}[subsection]
\newtheorem{Theorem}{Theorem}[subsection]
\newtheorem{Lemma}{Lemma}[subsection]
\newtheorem{prop}{Proposition}[subsection]
\newtheorem{remark}{Remark}[subsection]
\newtheorem{coro}{Corollary}[subsection]

\newcommand{\seqnum}[1]{\href{https://oeis.org/#1}{\rm \underline{#1}}}

\begin{document}
		\begin{center}
		\vskip 1cm{\LARGE\bf 
			Explicit Identities and new results for Infinite Series associated 	with the Ratio of Central 
			\vskip .05in
			Binomial Coefficients 
			}
		\vskip 1cm
		\large
	Yao Mawugna Dzokotoe \\ Ecole Nationale Sup\'erieure de Statistique et d' Economie Appliqu\'ee,\\ Abidjan, C\^ote d'Ivoire. \\  \href{mailto:mawugna.dzokotoe@gmail.com}{\tt mawugna.dzokotoe@gmail.com} 
	\vspace{3mm}\\	
	Segun Olofin Akerele
	\\ Department of Mathematics, University of Ibadan, Oyo State, Nigeria.\\
	\href{mailto:akereleolofin@gmail.com}{\tt akereleolofin@gmail.com}
	\end{center}
	
	\vskip .2 in
	
\begin{abstract}
	We investigate some classes of infinite series involving central binomial coefficients, particularly focusing on those arising from ratios such as $\binom{2n}{n}\binom{4n}{2n}^{-1}$,$\binom{4n}{2n}\binom{2n}{n}^{-1}$  and related expressions. We derive several new explicit identities and closed-form evaluations, building on and refining previous results by Bhandari (2022) and Adegoke et al.  (2022).
\end{abstract}
	
	\section{Introduction}
Infinite series involving central binomial coefficients have long captured the interest of mathematicians due to their rich combinatorial structure and deep connections to number theory, special functions, and analysis. The central binomial coefficient, given by
\begin{equation}
	\binom{2n}{n}=\frac{(2n)!}{(n!)^2}.
\end{equation}
for $n\geq 0$, appears prominently in numerous areas, including combinatorics, number theory and mathematical analysis, you can see \cite{gould,9,3,10,13}. They are considered as special sequences and are indexed as sequence \seqnum{A000984} in the On-Line Encyclopedia of Integer Sequences \cite{oeis} . Their generating function is given by: 
$$\sum_{n=0}^{\infty}\binom{2n}{n}x^n=\frac{1}{\sqrt{1-4x}}, \qquad \text{provided} \ |x|<\frac{1}{4}.$$
\par 	The work of Lehmer \cite{9} is an important reference in the subject of infinite series of central binomial coefficients. The current literature is extensive where series involving these numbers and other special sequences such as Harmonic numbers, Fibonacci numbers are studied. For example Boyadzhiev \cite[p.\ 2]{Boyadzhiev} obtained the following beautiful results
\begin{align*}
	&\sum_{n=0}^{\infty}\binom{2n}{n}\frac{(-1)^{n-1}}{4^n}H_n=\sqrt{2}\log\frac{2\sqrt{2}}{1+\sqrt{2}}.\\
	&\sum_{n=0}^{\infty}\binom{2n}{n}\frac{H_n}{8^n}=2\sqrt{2}\log\frac{1+\sqrt{2}}{2}.
\end{align*}
$H_n$ stands for the harmonic numbers defined by $H_n=\sum_{k=1}^n\frac{1}{k}$.\\
After Boyadzhiev, Chen \cite{Chen} obtained the following identities
\begin{align*}
	&\sum_{n=1}^{\infty}\binom{2n}{n}\frac{h_nF_n}{8^n}=\frac{1}{\sqrt{10}}\ln 2+\frac{1}{\sqrt{2}}\ln\left(\frac{3+\sqrt{5}}{2}\right).\\
	&\sum_{n=1}^{\infty}\binom{2n}{n}\frac{n^2h_n}{8^n}=\frac{3}{2}\sqrt{2}+\frac{5\sqrt{2}}{3}\ln 2.
\end{align*}	
where,
$$h_n=1+\frac{1}{3}+\frac{1}{5}+\cdots+\frac{1}{2n-1} $$
and $F_n$ denotes the $n-th$ Fibonacci number satisfying the recurrence relation $F_n=F_{n-1}+F_{n-2}$, $n\geq 2$ with conditions $F_0=0, F_1=1$.

Our motivation to write this paper started with a 2022 paper \cite{1} by Bhandari. The author used creative techniques based on building ordinary generating functions of central binomial coefficients and making use of Wallis' integral formulas to derive interesting series for $\binom{2n}{n} \binom{4n}{2n}$ and $\binom{2n}{n}\binom{4n}{2n}^{-1}$. As examples he obtained for $\binom{2n}{n}\binom{4n}{2n}$:
$$\sum_{n=0}^{\infty}\frac{1}{(n+1)64^n}\binom{2n}{n}\binom{4n}{2n}=\frac{8\sqrt{2}}{3\pi}.$$
and for $\binom{2n}{n}\binom{4n}{2n}^{-1}$ he obtained among others, these two beautiful identities:
\vspace{-3mm}
\begin{align*}
	&\sum_{n=0}^{\infty}\frac{4^n}{(2n-1)^2}\frac{\binom{2n}{n}}{\binom{4n}{2n}}=4(\sqrt{2}-1).\\
	&\sum_{n=1}^{\infty}\frac{n4^n}{(2n-1)^2(4n+1)}\frac{\binom{2n}{n}}{\binom{4n}{2n}}=\frac{5\sqrt{2}-4}{9}.
\end{align*}
\par Main theorems in \cite{1} involving $\binom{2n}{n}\binom{4n}{2n}^{-1}$ are expressed in terms of finite binomial sums of two families of integrals, namely:
$$\mathcal{B}(k)=\int_0^{1/2}\frac{t^k}{\sqrt{1-t}}\,\mathrm{d}t,	\quad \text{and} \quad \varphi(2k+1)=\int_0^1 t^{2k+1}\sqrt{1+t^2}\,\mathrm{d}t$$
for $k\in \mathbb{Z}_{\geq 0}$.\\
The determination of general closed form formulas for the integrals were left as open problems. The attempt to answer this question is the starting point of this work and we obtain:
\begin{align*}
	&\mathcal{B}(k)=\sum_{p=0}^k\binom{k}{p}\frac{(-1)^p}{2p+1}\left(2-\frac{\sqrt{2}}{2^p}\right).\\
	&\varphi(2k+1)=\sum_{p=0}^k\binom{k}{p}(-1)^{k-p}\frac{\sqrt{2}\cdot 2^{p+1}-1}{2p+3}.
\end{align*}
Recently, we came across the paper of Adegoke et al.\ \cite{2} which became the second source of motivation for writing this paper. In Section 7 of \cite{2} the authors gave two theorems on series involving $\binom{4n}{2n}\binom{2n}{n}^{-1}$. Once again these two theorems are left as far of two integrals, $\int_0^{\frac{\pi}{2}}\sin^{2r}x\sin\frac{x}{2}\,\mathrm{d}x$ and $\int_0^{\frac{\pi}{2}}\sin^{2r-1}x\cos\frac{x}{2}\,\mathrm{d}x$. We investigated their closed forms and we obtained that:
\begin{align*}
	&\int_0^{\frac{\pi}{2}}\sin^{2r}x\sin\frac{x}{2}\,\mathrm{d}x=\frac{16^r}{(4r+1)\binom{4r}{2r}}\left(2+\sqrt{2}\sum_{k=0}^r \frac{\binom{4k}{2k}}{(4k-1)16^k}\right).\\
	&\int_0^{\frac{\pi}{2}}\sin^{2r-1}x\cos\frac{x}{2}\,\mathrm{d}x=\frac{6\cdot 16^{r-1}}{r\binom{4r}{2r}}\left(\frac{4}{3}-\sqrt{2}\sum_{k=0}^{r-1} \frac{\binom{4k}{2k}}{(6k+3)16^k}\right).
\end{align*}
It's worth noting that infinite series involving ratio of central binomial coefficients are rare in the existing literature and it would be handy to fill this gap. This work contributes to this endeavor by reviewing explicitly, theorems on $\binom{2n}{n}\binom{4n}{2n}^{-1}$ from \cite{1} and $\binom{4n}{2n}\binom{2n}{n}^{-1}$ from \cite{2} based on the aforementioned integrals evaluations. \\
Besides this we derive new interesting theorems from which we choose the following ones as showcase:
\begin{align*}
	&\sum_{n=1}^{\infty}\frac{1}{4^nn(2n+3)}\frac{\binom{4n-2}{2n-1}}{\binom{2n+2}{n+1}}=\frac{27\sqrt{2}-26}{420}.\\
	&\sum_{n=1}^{\infty}\frac{n^2}{(2n+3)(2n-1)^2(2n+1)}\frac{\binom{2n}{n}}{\binom{2n+2}{n+1}}=\frac{16+5\pi^2}{1024}.\\
	&\sum_{n=1}^{\infty}\frac{1}{(2n+5)(n+1)(2n+1)(2n+3)}\frac{\binom{2n}{n}}{\binom{2n+4}{n+2}}=\frac{5\pi^2}{1024}-\frac{137}{2880}.
\end{align*}
Throughout this paper, we verify our results using Computer Algebra System (CAS) software Mathematica 13.3.

\section{Preliminaries}\label{prelim}
In this section we start with some interesting generating functions from \cite{7,8} for $\binom{2n}{n}$ and $\binom{2n}{n}^{-1}$ along with their proof that will serve to derive new results in section 5.
\begin{Lemma}\label{lemma2.0.1}	
	For $|x|\leq 1$,
	\begin{equation}
		\label{1}
		\sum_{n=1}^{\infty}\frac{nx^{2n+3}}{4^n(2n-1)^2(2n+1)(2n+3)}\binom{2n}{n}=\frac{(8x^4-8x^2+3)\sin^{-1}x+\sqrt{1-x^2}(6x^3-3x)}{128}.
	\end{equation}
	\begin{proof}
		It is easy to show that,
		\begin{equation}
			\label{buhari}
			\sum_{n=1}^{\infty}\frac{nx^{2n}}{4^n(2n-1)^2(2n+1)}\binom{2n}{n}=\frac{1}{8}\left(\sqrt{1-x^2}+2x\sin^{-1}x-\frac{\sin^{-1}x}{x}\right).
		\end{equation}
		Next we multiply (\ref{buhari}) by $x^2$ and integrate both sides with respect to $x$ to get the result.
	\end{proof}	
\end{Lemma}	
\begin{Lemma}\label{lemma2.0.2}
	For $|x|\leq 1$ and $x\neq 0$,
	\begin{equation}
		\label{2}
		\sum_{n=1}^{\infty}\frac{2n^2x^{2n+1}}{4^n(2n-1)^2(2n+1)}\binom{2n}{n}=\frac{(2x^2+1)\sin^{-1}x-x\sqrt{1-x^2}}{8}.
	\end{equation}
\end{Lemma}	
\begin{proof}
	The result follows immediately by differentiating \eqref{buhari} with respect to $x$.
\end{proof}
\begin{Lemma}\label{lemma2.0.3}
	For all $x\in [0,4)$,
	\begin{equation}
		\label{3}
		\sum_{n=1}^{\infty}\frac{x^{n+3/2}}{n(n+3/2)\binom{2n}{n}}=\frac{4}{9}\sqrt{4-x}\left(\sqrt{\frac{x}{4-x}}(x+24)-3(x+8)\arctan\sqrt{\frac{x}{4-x}}\right).
	\end{equation}
\end{Lemma}	
\begin{proof}
	Observe that,
	\begin{align*}
		&\sum_{n=1}^{\infty}\frac{x^n}{n\binom{2n}{n}}=\frac{1}{2}\sum_{n=1}^{\infty}x^n\int_0^{\frac{\pi}{2}}2\sin^{2n-1}\theta\cos^{2n-1}\theta\,\mathrm{d}\theta=\int_0^{\frac{\pi}{2}}\frac{1}{\sin\theta\cos\theta}\sum_{n=1}^{\infty}(x\sin^2\theta\cos^2\theta)^n\,\mathrm{d}\theta\\
		&=\int_0^{\frac{\pi}{2}}\frac{x\sin\theta\cos\theta}{1-x\sin^2\theta\cos^2\theta}\,\mathrm{d}\theta = \frac{x}{2}\int_0^{\frac{\pi}{2}}\frac{\sin 2\theta}{1-\frac{x}{4}\sin 2\theta}\,\mathrm{d}\theta.
	\end{align*}
	by the change of variable $p=\cos 2\theta$ , we could easily verify that
	$$	\sum_{n=1}^{\infty}\frac{x^n}{n\binom{2n}{n}}=2\sqrt{\frac{x}{4-x}}\arctan\sqrt{\frac{x}{4-x}}$$
	From the above identity, multiply both sides by $\sqrt{x}$ and then integrate with respect to $x$. Thus, the result follows.
\end{proof}
\begin{Lemma}\label{lemma2.0.4}
	For all $x\in [-1,1]$ and $x\neq 0$, 
	\begin{equation}
		\label{4}
		\sum_{n=1}^{\infty}\binom{2n}{n}\frac{x^{2n+2}}{2^{2n+1}(n+1)(2n+3)}=1-\frac{x^2}{6}-\frac{\sqrt{1-x^2}}{2}-\frac{\arcsin x}{2x}.
	\end{equation}
\end{Lemma}	
\begin{proof}
	It can be shown that $$\sum_{n=1}^{\infty}\binom{2n}{n}\frac{x^{2n}}{4^n}=\frac{1}{\sqrt{1-x^2}}-1$$
	Thus, multiply both sides of the above identity by $x^2$ and then integrate both sides to get,
	$$\sum_{n=1}^{\infty}\binom{2n}{n}\frac{x^{2n+3}}{(2n+3)4^n}=\frac{1}{6}(3\arcsin x-3x\sqrt{1-x^2}-2x^3)+C$$
	Observe, as $x\to 0$ we have that $C\to 0$. Hence $C=0$. By dividing through by $x^2$ and integrating both sides, we obtain;
	$$	\sum_{n=1}^{\infty}\binom{2n}{n}\frac{x^{2n+2}}{2^{2n+1}(n+1)(2n+3)}=-\frac{x^2}{6}-\frac{\sqrt{1-x^2}}{2}-\frac{\arcsin x}{2x}+C_1$$
	Similarly as $x\to 0$, we have that $C_1\to 1$. Hence, the result follows directly.
\end{proof}

\begin{Lemma}\label{lemma2.0.6}
	For all $x\in [-1,1]$ and $x\neq 0$, 
	\begin{equation}
		\label{6}
		\sum_{n=1}^{\infty}\binom{2n}{n}\frac{x^{2n+2}}{2^{2n+1}(n+1)(2n+1)(2n+3)}=\frac{1}{12}\left(9\sqrt{1-x^2}+\frac{6x^2+3}{x}\arcsin x-2x^2-12\right).
	\end{equation}
\end{Lemma}	
\begin{proof}
	Notice, 
	$$\sum_{n=1}^{\infty}\binom{2n}{n}\frac{x^{2n+2}}{2^{2n+1}(n+1)(2n+1)}=x\arcsin x+\sqrt{1-x^2}-\frac{x^2}{2}-1$$
	The result follows from the above identity.
\end{proof}
For convenience we recall the well known \textit{Wallis Integral} formulas that will serve throughout this paper.
\begin{Lemma}[Wallis Integrals]\label{lemma2.0.7}
	For any non-negative integer $n$,
	\begin{align}
		&\int_{0}^{\frac{\pi}{2}}\sin^{2n}x\,\mathrm{d}x=\frac{\pi}{2}\binom{2n}{n}\frac{1}{4^n},\\
		&\int_{0}^{\frac{\pi}{2}}\sin^{2n+1}x\,\mathrm{d}x=\frac{4^n}{(2n+1)\binom{2n}{n}}.
	\end{align}
\end{Lemma}	
In what follows, we present a key lemma that is highly useful for establishing several results. It states a general formula for linear recursive sequences of order 1 with non-constant coefficients.
\begin{Lemma}\label{lemma2.0.8}
	Let $a_n, b_n$ and $r_n$ be sequences with $a_n,b_n\neq 0$. Assume that $z_n$ satisfies 
	\begin{equation}
		a_nz_n=b_nz_{n-1}+r_n, \quad n\geq 1.
	\end{equation}
	with initial condition $z_0$. Then 
	\begin{equation}
		z_n=\frac{b_1b_2\cdots b_n}{a_1a_2\cdots a_n}\left(z_0+\sum_{k=1}^n\frac{a_1a_2\cdots a_{k-1}}{b_1b_2\cdots b_k}r_k\right).
	\end{equation}
	
\end{Lemma}
\begin{proof}
	A proof was stated in \cite[p.\ 43]{3}.
\end{proof}
V. Moll widely used this result in \cite{6} to obtain several closed form for integrals, in particular some trigonometric integrals \cite[Chapter 5]{6}. This result, yet practical, seems to be unknown. It appears in the book of Andrica and Bagdasar \cite[p.\ 21]{6} in the following form:\\
Let $(a_n)_{n\geq 0}$ and $(b_n)_{n\geq 0}$ be two given sequences of complex numbers, and let $(x_n)_{n\geq 0}$  be  the sequence defined by 
$x_{n+1}=a_nx_n+b_n,  n=0,1,2,\dots,$
where $x_0=\alpha$ is a complex number. We have the following result.
If $a_n\neq 0$ for every $n\neq 0$, then the following formula holds:
\begin{equation}
	x_n=a_0\cdots a_{n-1}\left(\alpha+\sum_{k=0}^{n-1}\frac{b_k}{a_0\cdots a_k}\right), \qquad n=0,1,2,\dots,
\end{equation}
\section{Evaluation of key integrals and refinement of main results on series involving \texorpdfstring{$\binom{2n}{n}\binom{4n}{2n}^{-1}$}{binomial(2n,n)binomial(4n,2n)} arising from  \texorpdfstring{\cite{1}}{1}} 	
Consider the following integrals,
$$\mathcal{B}(k)=\int_0^{1/2}\frac{t^k}{\sqrt{1-t}}\,\mathrm{d}t,\quad  \varphi(2k)=\int_0^1 t^{2k}\sqrt{1+t^2}\,\mathrm{d}t	\quad \text{and} \quad \varphi(2k+1)=\int_0^1 t^{2k+1}\sqrt{1+t^2}\,\mathrm{d}t$$
for $k\in \mathbb{Z}_{\geq 0}$.	\\
They appeared in the formulas of many theorems in \cite{1}. In fact the author added the following comment about $\mathcal{B}(k)$ and $\varphi(2k+1)$ in \cite[p.\ 17]{1}: ``\emph{It is easy to deduce the primitives of the integrals for some particular values of $k$ and $m$; however, we do not have compact closed forms in terms of elementary functions that can easily generate an infinite number of solutions.''}\\
In this section, we start by answering this question by providing closed form for $\mathcal{B}(k)$, $\varphi(2k+1)$ and $\varphi(2k)$ too. We then use the obtained results for $\mathcal{B}(k)$, $\varphi(2k+1)$ to restate main theorems in \cite{1} involving $\binom{2n}{n}\binom{4n}{2n}^{-1}$.
\begin{prop}\label{prop2.1.1}
	For $k\in \mathbb{Z}_{\geq 0}$,
	$$(\mathcal{B}(k))=\int\frac{t^k}{\sqrt{1-t}}\,\mathrm{d}t=-\sum_{p=0}^k\binom{k}{p}\frac{(-1)^p}{p+\frac{1}{2}}(1-t)^{p+\frac{1}{2}}$$
\end{prop}
\begin{proof}
	Set $x=1-t$.
	So,\\
	\vspace{-3mm}
	$$\left(\mathcal{B}(k) \right)=-\int\frac{(1-x)^k}{\sqrt{x}}\,\mathrm{d}x$$
	By Binomial expansion, we have $(1-x)^k=\sum_{p=0}^k\binom{k}{p}(-1)^px^p$. So;
	\vspace{-3mm}
	$$\frac{(1-x)^k}{\sqrt{x}}=\sum_{p=0}^k\binom{k}{p}(-1)^px^{p-\frac{1}{2}}$$
	Which implies that,
	\vspace{-2mm}
	\begin{align*}
		\left(\mathcal{B}(k) \right)&=-\int \sum_{p=0}^k\binom{k}{p}(-1)^px^{p-\frac{1}{2}}\,\mathrm{d}x=-\sum_{p=0}^k\binom{k}{p}(-1)^p\int x^{p-\frac{1}{2}}\,\mathrm{d}x\\
		&= -\sum_{p=0}^k\binom{k}{p}\frac{(-1)^p}{p+\frac{1}{2}}x^{p+\frac{1}{2}}
	\end{align*}
	with $x=1-t$, we have:
	\vspace{-3mm}
	$$\left(\mathcal{B}(k) \right)=-\sum_{p=0}^k\binom{k}{p}\frac{(-1)^p}{p+\frac{1}{2}}(1-t)^{p+\frac{1}{2}}$$	
\end{proof}
\begin{Theorem}\label{theorem3.0.1}
	For $k\in \mathbb{Z}_{\geq 0}$,
	\begin{equation}
		\label{thm 1}
		\mathcal{B}(k)=\int_0^{1/2}\frac{t^k}{\sqrt{1-t}}\,\mathrm{d}t=\sum_{p=0}^k\binom{k}{p}\frac{(-1)^p}{2p+1}\left(2-\frac{\sqrt{2}}{2^p}\right)
	\end{equation}
\end{Theorem}	
\begin{proof}
	We get using Proposition 2.0.1,
	$$\left(\mathcal{B}(k) \right)=-\sum_{p=0}^k\binom{k}{p}\frac{(-1)^p}{p+\frac{1}{2}}(1-t)^{p+\frac{1}{2}}$$
	Hence,
	$$\mathcal{B}(k)=\left(\mathcal{B}(k)\right)\bigg|^{\frac{1}{2}}_0 =-\sum_{p=0}^k\binom{k}{p}\frac{(-1)^p}{p+\frac{1}{2}}\left(\left(\frac{1}{2}\right)^{p+\frac{1}{2}}-1\right)$$
	and finally, we get,
	$$\mathcal{B}(k)=\sum_{p=0}^k\binom{k}{p}\frac{(-1)^p}{2p+1}\left(2-\frac{\sqrt{2}}{2^p}\right)$$
\end{proof}	
\begin{remark}  In \cite{1} the author displayed the following first values:
	
	$$	\mathcal{B}(0)=2-\sqrt{2}, \ 
	\mathcal{B}(1)=\frac{8-5\sqrt{2}}{6}, \ 
	\mathcal{B}(2)=\frac{64-43\sqrt{2}}{60},$$
	$$
	\mathcal{B}(3)=\frac{256-177\sqrt{2}}{280}, \ 
	\mathcal{B}(4)=\frac{4096-2867\sqrt{2}}{5040}.$$
\end{remark}
While working on the integral $\mathcal{B}(k)$ , the first author observed, based on the above particular cases that  $\mathcal{B}(k)=\nu_1(k)-\nu_2(k)\sqrt{2}$ where $\nu_1(k)$ and $\nu_2(k)$ are positive rational numbers. Trying to analyze the ratio $\nu_1(k)/\nu_2(k)$ for the first values of $k$, he noticed that the result gets close to $\sqrt{2}$ which suggested to him that
\begin{equation}
	\lim_{k\to +\infty}\frac{\nu_1(k)}{\nu_2(k)}=\sqrt{2}.
\end{equation}
This observation was later confirmed with the evaluation of $\mathcal{B}(k)$. Now since   
\begin{align*}
	\nu_1(k):=2\sum_{p=0}^k\binom{k}{p}\frac{(-1)^p}{2p+1} \qquad \text{and} \qquad 	\nu_2(k):=\sum_{p=0}^k\binom{k}{p}\frac{(-1/2)^p}{2p+1}.
\end{align*}
it is easy to prove that 
\begin{equation*}
	\lim_{k\to +\infty}\frac{\nu_1(k)}{\nu_2(k)}=\sqrt{2}.
\end{equation*}
\begin{prop}\label{prop2.1.2}
	For $k\in \mathbb{Z}_{\geq 0}$,
	$$(\varphi(2k+1))=\int t^{2k+1}\sqrt{1+t^2}\,\mathrm{d}t=\sum_{p=0}^k\binom{k}{p}(-1)^{k-p}\frac{\sqrt{1+t^2}(1+t^2)^{p+1}}{2p+3}$$
\end{prop}
\begin{proof}
	Substitute $t=\tan(x)$ \\
	Then, 
	$$\left(\varphi(2k+1)\right)=\int \tan^{2k+1}(x)\sec^3(x)\,\mathrm{d}x$$
	We can express $\tan^{2k+1}(x)$ in terms of $\sec(x)$ as follows;
	\begin{equation}
		\label{proof2}
		\tan^{2k}(x)=\left(\sec^2(x)-1\right)^k=\sum_{p=0}^k\binom{k}{p}(-1)^{k-p}\sec^{2p}(x)
	\end{equation}
	Using (\eqref{proof2}), we have that;
	\vspace{-2mm}
	$$\tan^{2k+1}(x)\sec^3(x)=\sum_{p=0}^k\binom{k}{p}(-1)^{k-p}\sec'(x)\sec^{2p+2}(x)$$
	where $\sec'(x)=\frac{\sin(x)}{\cos^2(x)}$ which implies
	\vspace{-3mm}
	\begin{align*}
		\left(\varphi(2k+1)\right) &=\int \tan^{2k+1}(x)\sec^3(x)\,\mathrm{d}x=\int \sum_{p=0}^k\binom{k}{p}(-1)^{k-p}\sec'(x)\sec^{2p+2}(x)\,\mathrm{d}x\\
		&= \sum_{p=0}^k\binom{k}{p}(-1)^{k-p}\int\sec'(x)\sec^{2p+2}(x)\,\mathrm{d}x\\
		&=\sum_{p=0}^k\binom{k}{p}(-1)^{k-p}\frac{\sec^{2p+3}(x)}{2p+3}
	\end{align*}
	Recall $t=\tan(x)$, we have $\sec(x)=\sqrt{1+t^2}$.\\
	It follows that:
	$$\left(\varphi(2k+1)\right)=\sum_{p=0}^k\binom{k}{p}(-1)^{k-p}\frac{\sqrt{1+t^2}(1+t^2)^{p+1}}{2p+3}$$	
\end{proof}
\begin{Theorem}\label{theorem3.0.2}
	For $k\in \mathbb{Z}_{\geq 0}$,
	\begin{equation}
		\label{thm2}
		\varphi(2k+1) = \int_{0}^1t^{2k+1}\sqrt{1+t^2}\,\mathrm{d}t=\sum_{p=0}^k\binom{k}{p}(-1)^{k-p}\frac{\sqrt{2}\cdot 2^{p+1}-1}{2p+3}
	\end{equation}
\end{Theorem}	
\begin{proof}
	Using Proposition 2.0.2
	Then:
	$$\varphi(2k+1) = \int_{0}^1t^{2k+1}\sqrt{1+t^2}\,\mathrm{d}t=\left(\varphi(2k+1)\right)\bigg|^1_0=\sum_{p=0}^k\binom{k}{p}(-1)^{k-p}\frac{\sqrt{2}\cdot 2^{p+1}-1}{2p+3}$$
\end{proof}
At this stage we are done with Bhandari's quest on $\mathcal{B}(k)$ and $\varphi(2k+1)$. Wondering about $\varphi(2k)$ for completeness we finally come out with the following result.
\begin{Theorem}\label{theorem2.1.3}
	For $k\in \mathbb{Z}_{\geq 0}$,
	\begin{equation}
		\label{thm3.0.3}
		\varphi(2k)=\int_0^1t^{2k}\sqrt{1+t^2}\,\mathrm{d}t=\frac{(-1)^k}{4^k} \frac{\binom{2k}{k}}{k+1}\cdot \left(I(0)+\sqrt{2}\sum_{p=1}^k\frac{(-1)^p4^p}{\binom{2p}{p}}\right).
	\end{equation}
	where $I(0)=\frac{\sqrt{2}}{2}+\frac{\ln(1+\sqrt{2})}{2}$.	
\end{Theorem}
\begin{proof}
	Let, 
	$$I(m):=\int_{0}^1 t^{2m}\sqrt{1+t^2}\,\mathrm{d}t =\varphi(2m)$$
	Then, 
	\vspace{-1mm}
	$$I(m+1)=\int_{0}^1 t^{2m+2}\sqrt{1+t^2}\,\mathrm{d}t$$
	\vspace{2mm}
	Using integration by parts, we have that;
	\begin{equation}
		\label{proof3}
		I(m+1)=\frac{\sqrt{2}}{2m+3}-\frac{1}{2m+3}\int_0^1\frac{t^{2m+4}}{\sqrt{1+t^2}}\,\mathrm{d}t
	\end{equation}
	We deduce from (\eqref{proof3}):
	\begin{equation}
		\label{proof3i}
		\frac{2m+1}{2m+3}I(m)=\frac{\sqrt{2}}{2m+3}-\frac{1}{2m+3}\int_0^1\frac{t^{2m+2}}{\sqrt{1+t^2}}\,\mathrm{d}t
	\end{equation}
	Adding both (\eqref{proof3}) and (\eqref{proof3i}) gives,
	\begin{align*}
		I(m+1)+\frac{2m+1}{2m+3}I(m) &=\frac{2\sqrt{2}}{2m+3}-\frac{1}{2m+3}\int_0^1\frac{t^{2m+2}(1+t^2)}{\sqrt{1+t^2}}\,\mathrm{d}t\\
		&= \frac{2\sqrt{2}}{2m+3}-\frac{1}{2m+3}\underbrace{\int_0^1 t^{2m+2}\sqrt{1+t^2}\,\mathrm{d}t}_{I(m+1)}
	\end{align*}
	Hence we deduce,
	\begin{equation}
		\label{proof3ii}
		I(m+1)=\frac{\sqrt{2}}{m+2}-\frac{m+\frac{1}{2}}{m+2}I(m), \quad m\geq 0
	\end{equation}
	For example $I(0)=\int_{0}^1\sqrt{1+t^2}\,\mathrm{d}t$, thus;
	$$I(0)=\int_{0}^1\sqrt{1+t^2}\,\mathrm{d}t=\frac{1}{2}(\sqrt{2}+\sinh^{-1}(1))=\frac{\sqrt{2}}{2}+\frac{\ln(1+\sqrt{2})}{2}.$$
	using $\sinh^{-1}(t)=\ln(t+\sqrt{1+t^2}).$
	\par With the relation (\eqref{proof3ii}) connecting $I(m+1)$ to $I(m)$ and the expression of $I(0)$, we find a closed form using Lemma 2.0.8,
	$$I(m)=\frac{(-1)^m}{4^m} \frac{\binom{2m}{m}}{m+1}\cdot \left(\frac{\sqrt{2}}{2}+\frac{\ln(1+\sqrt{2})}{2}+\sqrt{2}\sum_{k=1}^m\frac{(-1)^k4^k}{\binom{2k}{k}}\right).$$	
\end{proof}
\par Now, with the closed forms of $\mathcal{B}(k)$ and $\varphi(2k+1)$ for $k\in \mathbb{Z}_{\geq 0}$ obtained in this section, we revise Theorem 5, Theorem 6 and Theorem 9 in Bhandari's paper \cite{1}, as follows;
\begin{Theorem}\label{theorem2.1.4}
	For $r\in \mathbb{Z}_{\geq 1}$,
	\begin{align*}
		\sum_{n=1}^{\infty}\frac{n4^n}{(2n-1)^2(4n+2r-1)}\frac{\binom{2n}{n}}{\binom{4n+2r-2}{2n+r-1}}=\sqrt{2}\sum_{k=0}^{r}\frac{(-1)^k}{(2k+1)2^{2r-k-1}}\binom{r}{k}\mathcal{B}(k).
	\end{align*}	
\end{Theorem}
\begin{proof}
	Consider \cite[Thm.\ 5]{1}, put $m=2r-1$.
\end{proof}
\begin{Theorem}\label{theorem2.1.5}
	For $r\in \mathbb{Z}_{\geq 1}$, 
	\begin{align*}
		&\sum_{n=1}^{\infty}\frac{n4^n}{(2n-1)^2(2n+1)(4n+2r+1)}\frac{\binom{2n}{n}}{\binom{4n+2r}{2n+r}}=\frac{\sqrt{2}}{2}\sum_{k=0}^{r+1}\frac{(-1)^k}{(2k+1)2^{2r-k+1}}\binom{r+1}{k}\mathcal{B}(k)\\
		&+\frac{\varphi(2r+1)}{2^{2r+3}}-\frac{1}{2^{3/2}}\sum_{k=0}^{r-1}\frac{(-1)^k}{(2k+1)2^{2r-k+1}}\binom{r-1}{k}\mathcal{B}(k).		
	\end{align*}
\end{Theorem}
\begin{proof}
	Consider \cite[Thm.\ 6]{1}, put $m=2r+1$ and using Theorem \ref{theorem2.1.4}
\end{proof}
\begin{ex}
	\begin{align*}
		&\sum_{n=1}^{\infty}\frac{n4^n}{(2n-1)^2(2n+1)(4n+3)}\frac{\binom{2n}{n}}{\binom{4n+2}{2n+1}}=\frac{2+2\sqrt{2}}{225}.\\
		&\sum_{n=1}^{\infty}\frac{n4^n}{(2n-1)^2(2n+1)(4n+5)}\frac{\binom{2n}{n}}{\binom{4n+4}{2n+2}}=\frac{137\sqrt{2}-88}{22050}.
	\end{align*}
\end{ex}
\begin{Theorem}\label{theorem2.1.6}
	For $r\in \mathbb{Z}_{\geq 1}$,
	\begin{align*}
		\sum_{n=1}^{\infty}\frac{n4^n}{(4n^2-1)(4n+2r+1)}\frac{\binom{2n}{n}}{\binom{4n+2r}{2n+r}}=\frac{1}{\sqrt{2}}\sum_{k=0}^{r-1}\frac{(-1)^k}{(2k+1)2^{2r-k+1}}\binom{r-1}{k}\mathcal{B}(k)-\frac{\varphi(2r+1)}{2^{2r+2}}.
	\end{align*}
\end{Theorem}
\begin{proof}
	Consider \cite[Thm.\ 9]{1}, put $m=2r+1$.
\end{proof}
\begin{ex}
	\begin{align*}
		&\sum_{n=1}^{\infty}\frac{n4^n}{(4n^2-1)(4n+3)}\frac{\binom{2n}{n}}{\binom{4n+2}{2n+1}}=\frac{7\sqrt{2}-8}{60}.\\
		&\sum_{n=1}^{\infty}\frac{n4^n}{(4n^2-1)(4n+5)}\frac{\binom{2n}{n}}{\binom{4n+4}{2n+2}}=\frac{71\sqrt{2}-64}{5040}.		
	\end{align*}
\end{ex}
Using Lemmas [\ref{lemma2.0.1}, \ref{lemma2.0.2}, \ref{lemma2.0.4}, \ref{lemma2.0.6}] and following the same approach as in \cite{1} we can deduce interesting similar results in terms of $\mathcal{B}(k)$ and $\varphi(2k+1)$ like Theorems \ref{theorem2.1.4} - \ref{theorem2.1.6}. For brevity details are left to the interested readers.

\section{Reviewing two theorems on series involving \texorpdfstring{$\binom{4n}{2n}\binom{2n}{n}^{-1}$}{binomial(2n,n)} arising from \texorpdfstring{\cite{2}}{2}}
\par In the last section of \cite{2}, the authors discussed some combinatorial identities for finite sums and interesting infinite series involving central binomial coefficients. Among them, our focus here are on those containing ratio of central binomial coefficients. \cite[Thm.\ 9, Thm.\ 10, Thm.\ 11]{2}.  Only Theorem 9 was expressed in explicit closed-form.
\par Theorem 10 and Theorem 11 are left in terms of two trigonometric integrals.
Recall that these two theorems \cite[Thm.\ 10, Thm.\ 11]{2} are based on the following two generating functions:
\begin{Lemma}[\cite{2}, p.\ 2, p.\ 4]
	For all $z\neq 0$ and $|z|\leq 1$, 
	\begin{equation}
		\label{7}
		G_1(z):=\sum_{n=1}^{\infty}\frac{C_{2n-1}}{4^{2n-1}}z^{2n-1}=\frac{2}{z}-\frac{\sqrt{1+z}+\sqrt{1-z}}{z}.
	\end{equation}
	Putting $z=\sin y$, we have the trigonometric version 
	\begin{equation}
		G_{1t}(y):=\sum_{n=1}^{\infty}C_{2n-1}\frac{\sin^{2n-1}y}{4^{2n-1}}=\frac{4\sin^2\left(\frac{y}{4}\right)}{\sin y},\qquad |y|\leq \frac{\pi}{2}
	\end{equation}
\end{Lemma}	
\begin{Lemma}[\cite{2}, p.\ 2, p.\ 4]
	For all $z\neq 0$ and $|z|\leq 1$, 
	\begin{equation}
		\label{8}
		G_2(z):=	\sum_{n=1}^{\infty}\frac{C_{2n}}{4^{2n}}z^{2n}=\frac{\sqrt{1+z}-\sqrt{1-z}}{z}.
	\end{equation}
	This also appears as a Lemma in \cite{1}. Putting $z=\sin y$, yields the trigonometric version namely:
	\begin{equation}
		G_{2t}(y):=\sum_{n=1}^{\infty}C_{2n}\frac{\sin^{2n}y}{4^{2n}}=\frac{1}{\cos \left(\frac{y}{2}\right)},\qquad |y|\leq \frac{\pi}{2}
	\end{equation}
\end{Lemma}	
We now state a key lemma that will be of help for simplification.
\begin{Lemma}\label{lemma2.2.3}
	For all positive integers $p,n$, we have that 
	\begin{equation}
		\prod_{j=1}^p \left(\frac{j}{p}\right)_n = \frac{(pn)!}{p^{pn}}.
	\end{equation}
	where $(x)_n$ is the Pochammer symbol defined by $(x)_n=x(x+1)(x+2)\cdots (x+n-1)$, with $(x)_0=1$.
\end{Lemma}
It's worth noting that this Lemma appears in \cite[p.\ 10]{5}.
\begin{proof}
	Observe,
	\begin{align*}
		&\prod_{j=1}^p\left(\frac{j}{p}\right)_n=\prod_{j=1}^p\left(\frac{j}{p}\right)\left(\frac{j}{p}+1\right)\left(\frac{j}{p}+2\right)\cdots \left(\frac{j}{p}+n-1\right)\\
		&=\prod_{j=1}^p\frac{j(j+p)(j+2p)\cdots(j+p(n-1))}{p^n}=\frac{\prod_{j=1}^p j(j+p)(j+2p)\cdots(j+p(n-1))}{p^{pn}}\\
		&=\frac{(np)!}{p^{pn}}.
	\end{align*}
	
\end{proof}
In particular: 
\begin{align}
	\label{227}
	\frac{1}{n!}\left(\frac{1}{2}\right)_n&=\frac{1}{2^{2n}}\binom{2n}{n}.\\
	\label{228}
	\frac{1}{(n!)^2}\left(\frac{1}{4}\right)_n\left(\frac{3}{4}\right)_n&=\frac{1}{4^{3n}}\binom{4n}{2n}\binom{2n}{n}.
\end{align}

We begin by deriving closed forms for the integrals that appeared in Theorem 10 and 11 from \cite{2}.

\begin{prop}\label{prop2.2.1}
	For $r\in \mathbb{N}$, 
	\begin{equation}
		\int_0^{\frac{\pi}{2}}\sin^{2r}x\sin\frac{x}{2}\,\mathrm{d}x=\frac{16^r}{(4r+1)\binom{4r}{2r}}\left(2+\sqrt{2}\sum_{k=0}^r \frac{\binom{4k}{2k}}{(4k-1)16^k}\right)
	\end{equation}
\end{prop}
\begin{proof}
	Let $$K(r)=\int_0^{\frac{\pi}{2}}\sin^{2r}x\sin\frac{x}{2}\,\mathrm{d}x, \qquad r\geq 1.$$
	Using Integration by parts, we arrive at;
	$$K(r)=-\sqrt{2}+4r\int_{0}^{\frac{\pi}{2}}\cos\frac{x}{2}\cos x\sin^{2r-1}x\,\mathrm{d}x.$$
	To this end, we have that,
	\begin{align*}
		\int_{0}^{\frac{\pi}{2}}\cos\frac{x}{2}\cos x\sin^{2r-1}x\,\mathrm{d}x&=-2\int_{0}^{\frac{\pi}{2}}\sin\frac{x}{2}[-\sin^{2r}x+(2r-1)(1-\sin^2x)\sin^{2r-2}x]\,\mathrm{d}x\\
		&=-2\int_{0}^{\frac{\pi}{2}}\sin\frac{x}{2}[-2r\sin^{2r}x+(2r-1)\sin^{2r-2}x]\,\mathrm{d}x\\
		&=4rK(r)-2(2r-1)K(r-1).
	\end{align*}
	Thus,
	$$K(r)=-\sqrt{2}+16r^2K(r)-8r(2r-1)K(r-1).$$
	which can be re-written as 
	$$\left(r-\frac{1}{4}\right)\left(r+\frac{1}{4}\right)K(r)=r\left(r-\frac{1}{2}\right)K(r-1)+\frac{\sqrt{2}}{16}.$$
	Using Lemma \ref{lemma2.0.8} with $a_k=\left(k-\frac{1}{4}\right)\left(k+\frac{1}{4}\right)$, $b_k =k\left(k-\frac{1}{2}\right)$ and $r_k=\frac{\sqrt{2}}{16}$ we have that, 
	$$a_1\cdots a_r =\left(\frac{3}{4}\right)_r\left(\frac{5}{4}\right)_r \qquad \text{and} \qquad b_1\cdots b_r = n!\left(\frac{1}{2}\right)_r$$
	
	This implies that, 
	$$\begin{cases}
		a_1\cdots a_r=(4r+1)\left(\frac{1}{4}\right)_r\left(\frac{3}{4}\right)_r,\\
		b_1\cdots b_r = r!\left(\frac{1}{2}\right)_r.
	\end{cases}$$
	Using Lemma \ref{lemma2.2.3} in particular \eqref{228}, we have
	\begin{align*}
		a_1\cdots a_r&=(4r+1)\frac{(r!)^2}{4^{3r}}\binom{4r}{2r}\binom{2r}{r},\\
		b_1\cdots b_r&=\frac{(r!)^2}{4^{r}}\binom{2r}{r}.
	\end{align*}
	Hence,
	$$\frac{b_1b_2\cdots b_r}{a_1a_2\cdots a_r}=\frac{16^r}{(4r+1)\binom{4r}{2r}}$$
	and
	$$\frac{a_1a_2\cdots a_{k-1}}{b_1b_2\cdots b_k}=\frac{a_1a_2\cdots a_{k}}{b_1b_2\cdots b_k}\cdot\frac{1}{a_k}=\frac{\binom{4k}{2k}}{16^{k-1}(4k-1)}.$$
	Finally, the result follows immediately.
\end{proof}
\begin{prop}\label{prop2.2.2}
	For $r\in \mathbb{N}$, 
	\begin{equation}
		\int_0^{\frac{\pi}{2}}\sin^{2r-1}x\cos\frac{x}{2}\,\mathrm{d}x=\frac{6\cdot 16^{r-1}}{r\binom{4r}{2r}}\left(\frac{4}{3}-\sqrt{2}\sum_{k=0}^{r-1} \frac{\binom{4k}{2k}}{(6k+3)16^k}\right)
	\end{equation}
\end{prop}
\begin{proof}
	Let
	$$I(r)=	\int_0^{\frac{\pi}{2}}\sin^{2r-1}x\cos\frac{x}{2}\,\mathrm{d}x, \qquad r\geq 1.$$
	It is immediate that $I(r)$ satisfies: 
	$$\left(r-\frac{1}{4}\right)\left(r-\frac{3}{4}\right)I(r)=(r-1)\left(r-\frac{1}{2}\right)I(r-1)-\frac{\sqrt{2}}{16}.$$
	Now, set $L(r)=I(r+1)$ for $r\geq 0$. Then 
	$$\left(r+\frac{1}{4}\right)\left(r+\frac{3}{4}\right)L(r)=r\left(r+\frac{1}{2}\right)L(r-1)-\frac{\sqrt{2}}{16}.$$
	Using Lemma \ref{lemma2.0.8}. with $a_k=\left(k+\frac{1}{4}\right)\left(k+\frac{3}{4}\right)$, $b_k =k\left(k+\frac{1}{2}\right)$ and $r_k=-\frac{\sqrt{2}}{16}$ we have that,
	\begin{align*}
		a_1\cdots a_r&=\left(1+\frac{1}{4}\right)_r\left(1+\frac{3}{4}\right)_r=\frac{(4r+1)(4r+3)}{3}\left(\frac{1}{4}\right)_r\left(\frac{3}{4}\right)_r\\
		b_1\cdots b_r&=r!\left(\frac{3}{2}\right)_r=(2r+1)r!\left(\frac{1}{2}\right)_r.	
	\end{align*}
	Using Lemma \ref{lemma2.2.3} in particular \eqref{227}, we have
	\begin{align*}
		a_1\cdots a_r&=\frac{(4r+1)(4r+3)}{3}\frac{(r!)^2}{4^{3r}}\binom{4r}{2r}\binom{2r}{r},\\
		b_1\cdots b_r&=(2r+1)\frac{(r!)^2}{4^{r}}\binom{2r}{r}.
	\end{align*}
	Hence,
	$$\frac{b_1b_2\cdots b_r}{a_1a_2\cdots a_r}=\frac{3(2r+1)16^r}{(4r+1)(4r+3)\binom{4r}{2r}}$$
	and
	$$\frac{a_1a_2\cdots a_{k-1}}{b_1b_2\cdots b_k}=\frac{a_1a_2\cdots a_{k}}{b_1b_2\cdots b_k}\cdot\frac{1}{a_k}=\frac{\binom{4k}{2k}}{16^{k-1}(6k+3)}.$$
	Finally, we obtain 
	$$I(r)=\frac{6\cdot 16^{r-1}}{r\binom{4r}{2r}}\left(\frac{4}{3}-\sqrt{2}\sum_{k=0}^{r-1} \frac{\binom{4k}{2k}}{(6k+3)16^k}\right)$$
\end{proof}
Now, provided the explicit formulas for these integrals, we obtain the following results.
\begin{Theorem}
	If $r$ is a non-negative integer, then 
	\begin{equation}
		\label{kunle1}
		\sum_{n=1}^{\infty}\frac{1}{(2n+1)(2n+2r+1)4^n}\frac{\binom{4n}{2n}}{\binom{2n+2r}{n+r}}=\frac{1}{2^{2r-1}}K(r)-\frac{1}{(2r+1)\binom{2r}{r}}.
	\end{equation}
	where,
	$$K(r)=\frac{16^r}{(4r+1)\binom{4r}{2r}}\left(2+\sqrt{2}\sum_{k=0}^r \frac{\binom{4k}{2k}}{(4k-1)16^k}\right)$$
\end{Theorem}
\begin{proof}
	Consider 
	$$G_{2t}(z):=\sum_{n=0}^{\infty}\frac{\binom{4n}{2n}}{(2n+1)4^n}\sin^{2n}z=\frac{1}{\cos\left(\frac{z}{2}\right)}, \qquad |z|\leq \frac{\pi}{2}.$$
	Multiply both sides by $\sin^{2r+1}z$ we obtain:
	$$\sum_{n=0}^{\infty}\frac{\binom{4n}{2n}}{(2n+1)4^n}\sin^{2n+2r+1}z=\frac{\sin^{2r+1}z}{\cos\left(\frac{z}{2}\right)}, \qquad |z|\leq \frac{\pi}{2}.$$
	Integrating both sides over $[0,\pi/2]$ and using Wallis formula; 
	$$\int_{0}^{\frac{\pi}{2}}\sin^{2n+2r+1}x\,\mathrm{d}x=\frac{4^{n+r}}{(2n+2r+1)\binom{2n+2r}{n+r}}$$
	The result follows immediately using Proposition \ref{prop2.2.1}.
\end{proof}
\begin{coro}
	For $r=0,1,2,3$, we have that
	\begin{align}
		\sum_{n=1}^{\infty}\frac{1}{4^n(2n+1)^2}\frac{\binom{4n}{2n}}{\binom{2n}{n}}&=3-2\sqrt{2},\\
		\sum_{n=1}^{\infty}\frac{1}{4^n(2n+1)(2n+3)}\frac{\binom{4n}{2n}}{\binom{2n+2}{n+1}}&=\frac{11-7\sqrt{2}}{30}, \\
		\sum_{n=1}^{\infty}\frac{1}{4^n(2n+1)(2n+5)}\frac{\binom{4n}{2n}}{\binom{2n+4}{n+2}}&=\frac{172-107\sqrt{2}}{2520}, \\
		\sum_{n=1}^{\infty}\frac{1}{4^n(2n+1)(2n+7)}\frac{\binom{4n}{2n}}{\binom{2n+6}{n+3}}&=\frac{6808-4175\sqrt{2}}{480480}.
	\end{align}
\end{coro}
\begin{Theorem}
	If $r$ is a positive integer, then 
	\begin{equation}
		\label{kunle2}
		\sum_{n=1}^{\infty}\frac{1}{n(2n+2r-1)4^n}\frac{\binom{4n-2}{2n-1}}{\binom{2n+2r-2}{n+r-1}}=\frac{1}{(2r-1)\binom{2r-2}{r-1}}-\frac{1}{2^{2r-2}}I(r), \qquad r\geq 1.
	\end{equation}	
	where,
	$$I(r)=\frac{6\cdot 16^{r-1}}{r\binom{4r}{2r}}\left(\frac{4}{3}-\sqrt{2}\sum_{k=0}^{r-1} \frac{\binom{4k}{2k}}{(6k+3)16^k}\right)$$
\end{Theorem}
\begin{proof}
	Consider, 
	$$G_{1t}(z):=\sum_{n=1}^{\infty}\frac{\binom{4n-2}{2n-1}}{2n}\frac{\sin^{2n-1}z}{4^{2n-1}}=\frac{4\sin^2\left(\frac{z}{4}\right)}{\sin z}, \qquad |z|\leq \frac{\pi}{2}.$$
	Multiply both sides by $\sin^{2r}z$, we obtain 
	$$\sum_{n=1}^{\infty}\frac{\binom{4n-2}{2n-1}}{2n}\frac{\sin^{2n+2r-1}z}{4^{2n-1}}=4\sin^{2r-1}z\sin^2\left(\frac{z}{4}\right), \qquad |z|\leq \frac{\pi}{2}.$$
	Now, using the fact that 
	$$2\sin^2\left(\frac{z}{4}\right)=1-\cos\left(\frac{z}{2}\right), $$
	It follows that,
	$$\sum_{n=1}^{\infty}\frac{\binom{4n-2}{2n-1}}{2n}\frac{\sin^{2n+2r-1}z}{4^{2n-1}}=2\sin^{2r-1}z-2\sin^{2r-1}z\cos\left(\frac{z}{2}\right), \qquad |z|\leq \frac{\pi}{2}.$$
	Integrate both sides over $[0,\pi/2]$ and using the Wallis' integral: 
	$$\int_{0}^{\frac{\pi}{2}}\sin^{2k+1}z\,\mathrm{d}z=\frac{4^k}{(2k+1)\binom{2k}{k}}.$$
	We obtain,
	$$	\sum_{n=1}^{\infty}\frac{1}{n(2n+2r-1)4^n}\frac{\binom{4n-2}{2n-1}}{\binom{2n+2r-2}{n+r-1}}=\frac{1}{(2r-1)\binom{2r-2}{r-1}}-\frac{1}{2^{2r-2}}\int_0^{\frac{\pi}{2}}\sin^{2r-1}z\cos\left(\frac{z}{2}\right)\,\mathrm{d}z.$$
	which appears in \cite[Thm.\ 11]{2}. Hence, using Proposition \ref{prop2.2.2}, the result follows immediately.
\end{proof}
\begin{coro}
	For $r=1,2,3$, we have that
	\begin{align}
		\sum_{n=1}^{\infty}\frac{1}{4^nn(2n+1)}\frac{\binom{4n-2}{2n-1}}{\binom{2n}{n}}&=\frac{\sqrt{2}-1}{3},\\
		\sum_{n=1}^{\infty}\frac{1}{4^nn(2n+3)}\frac{\binom{4n-2}{2n-1}}{\binom{2n+2}{n+1}}&=\frac{27\sqrt{2}-26}{420}, \\
		\sum_{n=1}^{\infty}\frac{1}{4^nn(2n+5)}\frac{\binom{4n-2}{2n-1}}{\binom{2n+4}{n+2}}&=\frac{755\sqrt{2}-712}{55440}.
	\end{align}
\end{coro}
\section{New Results on series involving \texorpdfstring{$\binom{2n}{n}\binom{2n+2r}{n+r}^{-1}$}{binomial(2n,n)}, \texorpdfstring{$r$}{r} a non negative parameter}
In this section, we present our main results. We start by a generalization of \cite[Thm.\ 9]{2} to 2 parameters. It's based on the following results obtained by Boyadzhiev along with its proof as it appears in \cite{Boyadzhiev}.
\begin{Lemma}[\cite{Boyadzhiev}, p.\ 9]
	For $m\in \mathbb{Z}_{\geq 0}$.
	\begin{equation}
		\label{section3}
		\sum_{n=0}^{\infty}\binom{2n}{n}\frac{x^{n+m+1}}{n+m+1}=\frac{1}{2^{2m+1}}\sum_{k=0}^m\binom{m}{k}\frac{(-1)^k}{2k+1}[1-(1-4x)^k\sqrt{1-4x}], \qquad |x|<\frac{1}{4}.
	\end{equation}
\end{Lemma}
\begin{proof}
	Let $f_m(x)$ be the L.H.S. Thus,
	$$\frac{\mathrm{d}(f_m(x))}{\mathrm{d}x}=\sum_{n=0}^{\infty}\binom
	{2n}{n}x^{n+m}=x^m\sum_{n=0}^{\infty}\binom{2n}{n}x^n=\frac{x^m}{\sqrt{1-4x}}.$$
	Using the well known generating function of central binomial coefficient
	\begin{equation}
		\sum_{n=0}^{\infty}\binom{2n}{n}x^n=\frac{1}{\sqrt{1-4x}}, \qquad |x|<\frac{1}{4}.
	\end{equation}
	So that, 
	$$f_m(x)=\int_0^x\frac{t^m}{\sqrt{1-4t}}\,\mathrm{d}t=-\frac{1}{2^{2m+1}}\int_0^{\sqrt{1-4x}}(1-y^2)^m\,\mathrm{d}y.$$
	using the substitution $y=\sqrt{1-4t}$ and expanding $(1-y^2)^m$ as\\ $(1-y^2)^m=\sum_{k=0}^m\binom{m}{k}(-1)^ky^{2k}$. We get,
	$$f_m(x)=	\frac{1}{2^{2m+1}}\sum_{k=0}^m\binom{m}{k}\frac{(-1)^k}{2k+1}[1-(1-4x)^k\sqrt{1-4x}].$$	
\end{proof}	
\begin{Theorem} For $m\in \mathbb{Z}_{\geq 0}.$
	\begin{align*}
		\label{section31}
		&\sum_{n=0}^{\infty}\frac{1}{(n+m+1)(2n+2m+2r+1)}\frac{\binom{2n}{n}}{\binom{2n+2m+2r}{n+m+r}}\\&=\frac{1}{2(2m+1)\binom{2m}{m}}\cdot\frac{1}{(2r-1)\binom{2r-2}{r-1}}-\frac{1}{4^{m+r}}\sum_{k=0}^m\binom{m}{k}\frac{(-1)^k}{2k+1}\frac{k!}{(r)_{k+1}}, \qquad r>0.
	\end{align*}
\end{Theorem}	
\begin{proof}
	Replacing $x$ by $x/4$ in \eqref{section3}, we obtain 
	\begin{equation}
		\label{proof1}
		\sum_{n=0}^{\infty}\binom{2n}{n}\frac{x^{n+m+1}}{4^{n+m+1}(n+m+1)}=\frac{1}{2\cdot 4^m}\sum_{k=0}^m\binom{m}{k}\frac{(-1)^k}{2k+1}[1-(1-x)^k\sqrt{1-x}], \qquad |x|<1.
	\end{equation}
	which implies that,
	\begin{align}
		\sum_{n=0}^{\infty}\binom{2n}{n}\frac{x^{n+m+1}}{4^n(n+m+1)}&=2\sum_{k=0}^m\binom{m}{k}\frac{(-1)^k}{2k+1}[1-(1-x)^k\sqrt{1-x}],\\
		&=2\sum_{k=0}^m\binom{m}{k}\frac{(-1)^k}{2k+1}-2\sum_{k=0}^m\binom{m}{k}\frac{(-1)^k}{2k+1}(1-x)^k\sqrt{1-x}.
	\end{align}	
	We remark the following fact;
	$$\sum_{k=0}^m\binom{m}{k}\frac{(-1)^k}{2k+1}=\int_0^1(1-t^2)^m\,\mathrm{d}t=\int_0^{\frac{\pi}{2}}\cos^{2m+1}x\,\mathrm{d}x.$$
	using the substitution $t=\sin x$. Thus, 
	$$	\sum_{k=0}^m\binom{m}{k}\frac{(-1)^k}{2k+1}=\frac{4^m}{(2m+1)\binom{2m}{m}}.$$
	We arrive at, 
	$$	\sum_{n=0}^{\infty}\binom{2n}{n}\frac{x^{n+m+1}}{4^n(n+m+1)}=\frac{2\cdot 4^m}{(2m+1)\binom{2m}{m}}-2\sum_{k=0}^m\binom{m}{k}\frac{(-1)^k}{2k+1}(1-x)^k\sqrt{1-x}.$$
	Let $x=\sin^2y$ and multiply both sides by $\sin^{2r-1}y$ for $r>0$. Then:
	\begin{align*}
		\begin{split}
			\sum_{n=0}^{\infty}\binom{2n}{n}\frac{\sin^{2n+2m+2r+1}y}{4^n(n+m+1)}=\frac{2\cdot 4^m}{(2m+1)\binom{2m}{m}}\sin^{2r-1}y-2\sum_{k=0}^m\binom{m}{k}\frac{(-1)^k}{2k+1}\cos^{2k+1}y\sin^{2r-1}y.
		\end{split}
	\end{align*}
	Integrating both sides over the interval $[0,\pi/2]$ and using the Wallis' integral    \eqref{lemma2.0.7} and the beta function, $B(u,v)=2\int_0^{\frac{\pi}{2}}\cos^{2u-1}\theta\sin^{2v-1}\theta\,\mathrm{d}\theta$ \cite[Section 5.12 ]{NIST}, we obtain:
	\begin{align*}
		\begin{split}
			\sum_{n=0}^{\infty}\frac{1}{(n+m+1)(2n+2m+2r+1)}\frac{\binom{2n}{n}}{\binom{2n+2m+2r}{n+m+r}}=\frac{1}{2(2m+1)\binom{2m}{m}}\cdot\frac{1}{(2r-1)\binom{2r-2}{r-1}}\\-\frac{1}{4^{m+r}}\sum_{k=0}^m\binom{m}{k}\frac{(-1)^k}{2k+1}B(k+1,r)
		\end{split}
	\end{align*}
	for $r>0$. 
	The result follows immediately from this end.	
\end{proof}
\begin{coro}\cite[p.\ 15, Theorem 9]{2}.
	If $r$ is a positive integer, then 
	\begin{align}
		\label{coro1}
		\sum_{n=0}^{\infty}\frac{1}{(n+1)(2n+2r+1)}\frac{\binom{2n}{n}}{\binom{2n+2r}{n+r}}=\frac{1}{2(2r-1)\binom{2r-2}{r-1}}-\frac{1}{2^{2r}r}.
	\end{align}
\end{coro}	
\par Motivated by the above corollary, we find several interesting similar results based on the generating functions presented in Section 2. Before we proceed we shall state a useful lemma required in the simplification of the expressions for the theorems that will come next.
\begin{Lemma}\label{lemma3.0.2}
	For $q\in \mathbb{Z}_{\geq 0},$
	$$\wp(q):=\int_0^{\frac{\pi}{2}}z\sin^qz\,\mathrm{d}z=\begin{cases}
		\vspace{4mm}
		\displaystyle{	\frac{\binom{2n}{n}}{4^{n+1}}\left(\frac{\pi^2}{2}+\sum_{k=1}^n\frac{4^k}{k^2\binom{2k}{k}}\right)}, & \text{if $q=2n$ },\\
		
		\displaystyle{	\frac{4^n}{(2n+1)\binom{2n}{n}}\left(1+\sum_{k=1}^n\frac{\binom{2k}{k}}{4^k(2k+1)}\right)}, & \text{if $q=2n+1$ }.
	\end{cases}$$
\end{Lemma}	
\begin{proof}
	From $\sin^2 z=1-\cos^2 z$, it is immediate that, 
	$$\wp(q)=\wp(q-2)-\int_0^{\frac{\pi}{2}}z\cos^2z\sin^{q-2}z\,\mathrm{d}z.$$
	using the Integration by parts formula, we observed that
	$$	\int_0^{\frac{\pi}{2}}z\cos^2z\sin^{q-2}z\,\mathrm{d}z=-\frac{1}{q(q-1)}+\frac{1}{q-1}\wp(q)$$
	Immediately, we see that 
	\begin{equation}
		\label{holy}
		\wp(q)=\frac{q-1}{q}\wp(q-2)+\frac{1}{q^2}.
	\end{equation}
	Let $x_n=\wp(2n)$ for $n\in \mathbb{Z}_{\geq 0}$, we have that 
	$$2nx_n=(2n-1)x_{n-1}+\frac{1}{4n^2},$$
	with $x_0=\wp(0)=\frac{\pi^2}{8},$ and using Lemma \ref{lemma2.0.8}, we get ,
	$$x_n=	\frac{\binom{2n}{n}}{4^{n+1}}\left(\frac{\pi^2}{2}+\sum_{k=1}^n\frac{4^k}{k^2\binom{2k}{k}}\right).$$
	Similarly, let $y_n=\wp(2n+1)$, we see that,
	$$(2n+1)y_n=2ny_{n-1}+\frac{1}{(2n+1)^2}.$$
	with $y_0=\wp(1)=1$, and using Lemma ~\ref{lemma2.0.8}, we get, 
	$$y_n=	\frac{4^n}{(2n+1)\binom{2n}{n}}\left(1+\sum_{k=1}^n\frac{\binom{2k}{k}}{4^k(2k+1)}\right).$$
\end{proof}
Before we continue, we would state some values for $\wp(n)$ as follows;
\begin{center}
	\begin{tabular}{|c|c|c|c|c|c|c|c|c|c|c|}
		\hline
		$n$& 1 & 2 & 3 & 4 & 5 & 6 & 7 & 8 & 9 & 10 \\
		\hline
		$\wp(n)$	& 1& $\frac{1}{4}+\frac{\pi^2}{16}$ & $\frac{7}{9}$ & $\frac{1}{4}+\frac{3\pi^2}{64}$ & $\frac{149}{225}$ & $\frac{17}{72}+\frac{5\pi^2}{128}$ & $\frac{2161}{3675}$ & $\frac{2}{9}+\frac{35\pi^2}{1024}$ & $\frac{53089}{99225}$ & $\frac{21}{100}+\frac{63\pi^2}{2048}$\\
		\hline
	\end{tabular}
\end{center}
\vspace{3mm}
Using the generating functions presented in section 2, Wallis integrals and the $\wp(q)$ formula in Lemma \ref{lemma3.0.2}, we derive new interesting series as it follows.	
\begin{Theorem}\label{theorem5.0.2}
	For every non-negative integer $r$,
	\begin{align*}
		&\sum_{n=1}^{\infty}\frac{n}{(2n+2r+3)(2n-1)^2(2n+1)(2n+3)}\frac{\binom{2n}{n}}{\binom{2n+2r+2}{n+r+1}}\\
		&=\frac{1}{2^{2r+2}}\left(\frac{8\wp(2r+4)-8\wp(2r+2)+3\wp(2r)}{128}+\frac{6}{128(4+2r)}-\frac{3}{128(2+2r)}\right).
	\end{align*}
\end{Theorem}
\begin{proof}
	From Lemma \ref{lemma2.0.1}, multiply through by $x^{2r}$ to get,
	\begin{align*}
		\sum_{n=1}^{\infty}\frac{nx^{2n+2r+3}}{4^n(2n-1)^2(2n+1)(2n+3)}\binom{2n}{n}&=\frac{(8x^{4+2r}-8x^{2+2r}+3x^{2r})\sin^{-1}x}{128}\\
		&+\frac{\sqrt{1-x^2}(6x^{3+2r}-3x^{1+2r})}{128}	.	
	\end{align*}
	Set $x=\sin t$ and integrate over the interval $[0,\frac{\pi}{2}]$ using the Wallis' integral, accompanied with Lemma~\ref{lemma3.0.2} to get the desired result.
\end{proof}
\begin{ex}For $r=0,1$, we have that
	\begin{align}
		&\sum_{n=1}^{\infty}\frac{n}{(2n+3)^2(2n-1)^2(2n+1)}\frac{\binom{2n}{n}}{\binom{2n+2}{n+1}}=\frac{\pi^2}{2048}.\\
		&\sum_{n=1}^{\infty}\frac{n}{(2n+5)(2n-1)^2(2n+1)(2n+3)}\frac{\binom{2n}{n}}{\binom{2n+4}{n+2}}=\frac{64+9\pi^2}{147456}.
	\end{align}
\end{ex}
Now, using similar approach, we can prove the following theorems below in conjunction to Lemma [\ref{lemma2.0.2}, \ref{lemma2.0.4}, \ref{lemma2.0.6}] respectively.
\begin{Theorem}
	\label{5.0.3}
	For every non-negative integer $r$,
	\begin{align*}
		\sum_{n=1}^{\infty}\frac{n^2}{(2n+2r+1)(2n-1)^2(2n+1)}\frac{\binom{2n}{n}}{\binom{2n+2r}{n+r}}=\frac{1}{2^{2r+1}}\left(\frac{2\wp(2r+2)+\wp(2r)}{8}-\frac{1}{16(r+1)}\right).
	\end{align*}
\end{Theorem}
\begin{ex}
	For $r=1,2$, we have that 
	\begin{align}
		&\sum_{n=1}^{\infty}\frac{n^2}{(2n+3)(2n-1)^2(2n+1)}\frac{\binom{2n}{n}}{\binom{2n+2}{n+1}}=\frac{16+5\pi^2}{2048}.\\
		&\sum_{n=1}^{\infty}\frac{n^2}{(2n+5)(2n-1)^2(2n+1)}\frac{\binom{2n}{n}}{\binom{2n+4}{n+2}}=\frac{40+9\pi^2}{18432}.
	\end{align}
\end{ex}	
\begin{Theorem}
	\label{5.0.4}
	For every non-negative integer $r$,
	\begin{align*}
		&\sum_{n=1}^{\infty}\frac{1}{(2n+2r+3)(n+1)(2n+1)(2n+3)}\frac{\binom{2n}{n}}{\binom{2n+2r+2}{n+r+1}}\\
		&=\frac{1}{2^{2r+1}}\left(\frac{3}{8(r+1)}+\frac{4\wp(2r+2)+2\wp(2r)}{8}\right)-\frac{1}{3(2r+3)\binom{2r+2}{r+1}}-\frac{1}{2(2r+1)\binom{2r}{r}}.
	\end{align*}
\end{Theorem}
\begin{ex}
	For $r=0,1$, we have that
	\begin{align}
		&\sum_{n=1}^{\infty}\frac{1}{(n+1)(2n+1)(2n+3)^2}\frac{\binom{2n}{n}}{\binom{2n+2}{n+1}}=\frac{9\pi^2-88}{288}.\\
		&\sum_{n=1}^{\infty}\frac{1}{(2n+5)(n+1)(2n+1)(2n+3)}\frac{\binom{2n}{n}}{\binom{2n+4}{n+2}}=\frac{5\pi^2}{1024}-\frac{137}{2880}.
	\end{align}
\end{ex}
\begin{Theorem}
	\label{5.0.5}
	For every non-negative integer $r$,
	\begin{align*}
		&\sum_{n=1}^{\infty}\frac{1}{(n+1)(2n+3)(2n+2r+3)}\frac{\binom{2n}{n}}{\binom{2n+2r+2}{n+r+1}}\\
		&=\frac{1}{2(2r+1)\binom{2r}{r}}-\frac{1}{3(2r+3)\binom{2r+2}{r+1}}-\frac{1}{2^{2r+1}}\left(\frac{1}{4(r+1)}+\frac{\binom{2r}{r}}{2\cdot 4^{r+1}}\left(\frac{\pi^2}{2}+\sum_{k=1}^r\frac{4^k}{k^2\binom{2k}{k}}\right)\right).
	\end{align*}
\end{Theorem}	
\begin{ex}
	For $r=0,1$, we have that 
	\begin{align*}
		&\sum_{n=1}^{\infty}\frac{1}{(n+1)(2n+3)^2}\frac{\binom{2n}{n}}{\binom{2n+2}{n+1}}=\frac{92-9\pi^2}{288}.\\
		&\sum_{n=1}^{\infty}\frac{1}{(n+1)(2n+3)(2n+5)}\frac{\binom{2n}{n}}{\binom{2n+4}{n+2}}=\frac{59}{1440}-\frac{\pi^2}{256}.
	\end{align*}
\end{ex}
\vspace{5mm}
\begin{remark}
	Observe that:
	\begin{align*}
		\binom{2n+2r}{n+r}&=\frac{(2n+2r)(2n+2r-1)\cdots(2n+r)\cdots(2n+2)(2n+1)(2n)!}{((n+r)(n+r-1)\cdots(n+1)(n!))^2}.\\
		&=\frac{2^r(2n+2r-1)(2n+2r-3)\cdots(2n+3)(2n+1)}{(n+r)(n+r-1)\cdots(n+2)(n+1)}\cdot\frac{(2n)!}{(n!)^2}.
	\end{align*}
	Thus, 
	\begin{equation}
		\label{remark}
		\binom{2n+2r}{n+r}=\frac{4^r(n+1/2)_r}{(n+1)_r}\binom{2n}{n}.
	\end{equation}
	in terms of the Pochhammer symbol $(x)_n=x(x+1)\cdots(x+n-1)$; $(x)_0=1$.
\end{remark}
Since Theorem \ref{theorem5.0.2}, Theorem \ref{5.0.3}, Theorem \ref{5.0.4} and Theorem \ref{5.0.5}  are of the form \\$\displaystyle{\sum_{n=0}^{\infty}\alpha(n)\frac{\binom{2n}{n}}{\binom{2n+r}{n+r}}}=\lambda(r)$, they can be rewritten using \eqref{remark} for simplifications:
\vspace{-3mm}
$$\sum_{n=0}^{\infty}\alpha(n)\frac{(n+1)_r}{(n+1/2)_r}=\mu(r).$$
Illustrating this, Theorem \ref{theorem5.0.2} has the following simplified form:
\begin{align}
	\begin{split}
		&\sum_{n=1}^{\infty}\frac{n}{(2n+2r+3)(2n-1)^2(2n+1)(2n+3)}\frac{(n+1)_{r+1}}{(n+1/2)_{r+1}}\\
		&=\frac{8\wp(2r+4)-8\wp(2r+2)+3\wp(2r)}{128}+\frac{6}{128(4+2r)}-\frac{3}{128(2+2r)}.
	\end{split}
\end{align}
In particular,
\begin{align}
	&\sum_{n=1}^{\infty}\frac{n(n+1)}{(n+1/2)(2n+1)(2n+3)^2(2n-1)^2}=\frac{\pi^2}{512}.\\
	&\sum_{n=1}^{\infty}\frac{n(n+1)(n+2)}{(n+1/2)(n+3/2)(2n+1)(2n+3)(2n+5)(2n-1)^2}=\frac{9\pi^2+64}{9216}.
\end{align}
\section{Concluding Remarks}
On the quest of finding closed forms for series of the form $\sum_{n=1}^{\infty}a(n)\binom{4n+2r}{2n+r}\binom{2n}{n}^{-1}$ where $r\geq 3$ and $a(n)$ is a sequence of real numbers, we were able to show using Lemma~\ref{lemma2.0.3} that,
\begin{align}\label{bounty}
	\begin{split}
		&\sum_{n=1}^{\infty}\frac{1}{4^nn(n+3/2)}\frac{\binom{4n+2r}{2n+r}}{\binom{2n}{n}}\\ &=\frac{4}{9}\binom{2r}{r}+\frac{128}{3}\binom{2r-4}{r-2}-\frac{4^{r+1}}{3\pi}\mathscr{F}(r-1)-\frac{2\cdot 4^{r+1}}{3\pi}\mathscr{F}(r-3).
	\end{split}
\end{align}
where,
$$\mathscr{F}(r):=\int_0^1t^{r-1/2}\sqrt{1+t}\arcsin t\,\mathrm{d}t\qquad r\geq 3.$$
In an earlier version of this paper, the authors did not succeed in evaluating $\mathscr{F}(0)$. It was then posted on Mathematics Stack Exchange by the second author and the MSE user, David. H \cite{mse}  brilliantly provided the closed form evaluation. We present his approach in the following Lemma.
\begin{Lemma}\label{baboy}
	If 
	$$\mathscr{F}(r):=\int_0^1 t^{r-1/2}\sqrt{1+t}\arcsin t\,\mathrm{d}t, \qquad \qquad \text{for} \ \  r\in \mathbb{N}\cup \{0\}.$$
	Then, 
	$$\mathscr{F}(0)=\frac{\pi}{2}\left(\sqrt{2}-1+\ln \left(\sqrt{2}+1\right)-\ln 2\right).$$
\end{Lemma}
\begin{proof}
	consider the following derivative:
	\begin{align*}
		\frac{\mathrm{d}}{\mathrm{d}t}\left[\sqrt{t(1+t)}\arcsin t\right]=\frac{\sqrt{t(1+t)}}{\sqrt{1-t^2}}+\frac{\sqrt{t}}{2\sqrt{1+t}}\arcsin t+\frac{\sqrt{1+t}}{2\sqrt{t}}\arcsin t.
	\end{align*}
	Integrating both sides and rearranging, we obtain the following integration by parts formula:
	$$\int_0^1\frac{\sqrt{1+t}}{\sqrt{t}}\arcsin t\, \mathrm{d}t=\pi\sqrt{2}-2\int_0^1\frac{\sqrt{t}}{\sqrt{1-t}}\,\mathrm{d}t-\int_0^1\frac{\sqrt{t}}{\sqrt{1+t}}\arcsin t\,\mathrm{d}t.$$
	Since, 
	\begin{align*}
		2\int_0^1\frac{\sqrt{t}}{\sqrt{1-t}}\,\mathrm{d}t=\pi \qquad \text{and} \qquad \int_0^1\frac{\sqrt{1+t}}{\sqrt{t}}\arcsin t\,\mathrm{d}t=\pi(\sqrt{2}-1)-\int_0^1\frac{\sqrt{t}}{\sqrt{1+t}}\arcsin t\,\mathrm{d}t.
	\end{align*}
	Recall that the inverse hyperbolic sine may be given by 
	$$\sinh^{-1}(z):=\int_0^z\frac{1}{\sqrt{1+y^2}}\,\mathrm{d}y=\int_0^1\frac{z}{\sqrt{1+z^2t^2}}\,\mathrm{d}t=\ln\left(z+\sqrt{1+z^2}\right); \qquad z\in \mathbb{R}.$$
	Observe,
	\begin{align*}
		\mathscr{I}=\int_0^1\frac{\sqrt{1+t}}{\sqrt{t}}\arcsin t\,\mathrm{d}t=\frac{\pi}{2}\left(\sqrt{2}-1\right)+\arcsin(1)\sinh^{-1}(1)-\int_0^1\frac{\sinh^{-1}(\sqrt{t})}{\sqrt{1-t^2}}\,\mathrm{d}t.
	\end{align*}
	Making the substitution $u=\sqrt{\frac{1-t}{1+t}}$ and integrating by parts, we find
	\begin{align*}
		\mathscr{I}&=\frac{\pi}{2}\left(\sqrt{2}-1\right)+\frac{\pi}{2}\ln\left(1+\sqrt{2}\right)-\int_0^1\int_0^1\frac{2\sqrt{2}u^2}{(1+x^2u^2)(1+u^2)\sqrt{1-u^2}}\,\mathrm{d}u\mathrm{d}x.\\
		&=\frac{\pi}{2}\left(\sqrt{2}-1\right)+\frac{\pi}{2}\ln\left(1+\sqrt{2}\right)-\pi\int_0^1 \frac{1}{1-x^2}\left[\frac{\sqrt{2}}{\sqrt{1+x^2}}-1\right]\,\mathrm{d}x.
	\end{align*}
	where we've made use of the following elementary integration formula to evaluate the integrals over $u$:
	$$\int_0^1\frac{2}{(1+r^2u^2)\sqrt{1-u^2}}\,\mathrm{d}u=\frac{\pi}{\sqrt{1+r^2}}; \qquad r\in \mathbb{R}.$$
	Observing that 
	$$1-x^2=\left[\sqrt{2}-\sqrt{1+x^2}\right]\left[\sqrt{2}+\sqrt{1+x^2}\right],$$
	we can use the Euler substitution $\sqrt{1+x^2}=x+y$ to rationalize the remaining integral for $\mathscr{I}$ and obtain the following result:
	\begin{align*}
		\mathscr{I}&=\frac{\pi}{2}\left(\sqrt{2}-1\right)+\frac{\pi}{2}\ln\left(1+\sqrt{2}\right)-\pi\int_0^1\frac{1}{\left[\sqrt{2}+\sqrt{1+x^2}\right]\sqrt{1+x^2}}\,\mathrm{d}x.\\
		&=\frac{\pi}{2}\left(\sqrt{2}-1\right)+\frac{\pi}{2}\ln\left(1+\sqrt{2}\right)-\pi\int_{\sqrt{2}-1}^1\frac{1}{\sqrt{2}+\left(\frac{1+y^2}{2y}\right)}\,\frac{\mathrm{d}y}{y}.\\
		&=\frac{\pi}{2}\left(\sqrt{2}-1\right)+\frac{\pi}{2}\ln\left(1+\sqrt{2}\right)-\pi\int_{\sqrt{2}-1}^1\frac{2}{y^2+2\sqrt{2}y+1}\,\mathrm{d}y.\\
		&=\frac{\pi}{2}\left(\sqrt{2}-1\right)+\frac{\pi}{2}\ln\left(1+\sqrt{2}\right)-\pi\int_{\sqrt{2}-1}^1\frac{2}{(y+\sqrt{2}-1)(y+\sqrt{2}+1)}\,\mathrm{d}y.\\
		&=\frac{\pi}{2}\left(\sqrt{2}-1\right)+\frac{\pi}{2}\ln\left(1+\sqrt{2}\right)-\pi\int_{\sqrt{2}-1}^1\frac{\mathrm{d}}{\mathrm{d}y}\ln\left(\frac{y+\sqrt{2}-1}{y+\sqrt{2}+1}\right)\,\mathrm{d}y.\\
		&=\frac{\pi}{2}\left(\sqrt{2}-1\right)+\frac{\pi}{2}\ln\left(1+\sqrt{2}\right)-\pi\left[\ln(\sqrt{2}-1)-\ln\left(\frac{\sqrt{2}-1}{\sqrt{2}}\right)\right].
	\end{align*}
	On simplification the result follows immediately.
\end{proof}
The above leads to the following proposition:
\begin{prop}
	For $r\in \mathbb{N}\cup \{0\}$.
	\begin{equation}
		\label{bountyqq}
		\mathscr{F}(r)=\frac{(-1)^r}{4^r}\frac{\binom{2r}{r}}{r+1}\left(\mathscr{F}(0)+\frac{1}{2}\sum_{k=1}^r\frac{(-1)^k4^k}{\binom{2k}{k}}\cdot\omega_k\right)
	\end{equation}
	where $\omega_k=2\sqrt{2}\pi-2B(k+3/2,1/2)-2B(k+1/2,1/2)$
\end{prop}
\begin{proof}
	Using Integration by part, we have:
	$$\mathscr{F}(r)=\frac{\sqrt{2}\pi}{2r+1}-\frac{1}{2r+1}\int_0^1\frac{t^{r+1/2}\arcsin t}{\sqrt{1+t}}\,\mathrm{d}t+\frac{2}{2r+1}\underbrace{\int_0^1\frac{t^{r+1/2}}{\sqrt{1-t}}\,\mathrm{d}t}_{B(r+3/2,1/2)}.$$
	where $B(m,n)$ is the Beta function. Thus,
	\begin{equation}
		\label{crazy1}
		(2r+1)\mathscr{F}(r)=\pi\sqrt{2}-2B(r+3/2,1/2)-\int_0^1\frac{t^{r+1/2}\arcsin t}{\sqrt{1+t}}\,\mathrm{d}t.
	\end{equation}
	Put $r\mapsto r-1$ in \eqref{crazy1}, we also have
	\begin{equation}
		\label{crazy2}
		(2r-1)\mathscr{F}(r-1)=\pi\sqrt{2}-2B(r+1/2,1/2)-\int_0^1\frac{t^{r-1/2}\arcsin t}{\sqrt{1+t}}\,\mathrm{d}t.	
	\end{equation}
	Adding both equations \eqref{crazy1} and \eqref{crazy2} implies,
	$$(2r+2)\mathscr{F}(r)=-(2r-1)\mathscr{F}(r-1)+2\sqrt{2}\pi-2B(r+3/2,1/2)-2B(r+1/2,1/2).$$
	using the key lemma \ref{lemma2.0.8} for recursive sequences, we get:
	\begin{equation*}\label{saver}
		\mathscr{F}(r)=\frac{(-1)^r}{4^r}\frac{\binom{2r}{r}}{r+1}\left(\mathscr{F}(0)+\frac{1}{2}\sum_{k=1}^r\frac{(-1)^k4^k}{\binom{2k}{k}}\cdot\omega_k\right)
	\end{equation*}
\end{proof}
Now, with the closed form evaluation of $\mathscr{F}(r)$, we formulate the following theorem:
\begin{Theorem} For $r\geq 3$ an integer, we have
	\begin{align*}
		\begin{split}
			&\sum_{n=1}^{\infty}\frac{1}{4^nn(n+3/2)}\frac{\binom{4n+2r}{2n+r}}{\binom{2n}{n}}\\ &=\frac{4}{9}\binom{2r}{r}+\frac{128}{3}\binom{2r-4}{r-2}-\frac{4^{r+1}}{3\pi}\mathscr{F}(r-1)-\frac{2\cdot 4^{r+1}}{3\pi}\mathscr{F}(r-3).
		\end{split}
	\end{align*}
	where $$\mathscr{F}(r)=\frac{(-1)^r}{4^r}\frac{\binom{2r}{r}}{r+1}\left(\mathscr{F}(0)+\frac{1}{2}\sum_{k=1}^r\frac{(-1)^k4^k}{\binom{2k}{k}}\cdot\omega_k\right)$$
\end{Theorem}

\begin{ex} For $r=3,4,5$ in  \eqref{bounty}, we have
	\begin{align}
		&\sum_{n=1}^{\infty}\frac{1}{4^nn(n+3/2)}\frac{\binom{4n+6}{2n+3}}{\binom{2n}{n}}=\frac{8}{9}\left(209-110\sqrt{2}+102\ln(2)-102\ln(1+\sqrt{2})\right).\\
		&\sum_{n=1}^{\infty}\frac{1}{4^nn(n+3/2)}\frac{\binom{4n+8}{2n+4}}{\binom{2n}{n}}=\frac{2}{9}\left(2407-1396\sqrt{2}-444\ln(2)+444\ln(1+\sqrt{2})\right).\\
		&\sum_{n=1}^{\infty}\frac{1}{4^nn(n+3/2)}\frac{\binom{4n+10}{2n+5}}{\binom{2n}{n}}=\frac{4}{15}\left(4537-1948\sqrt{2}+780\ln(2)-780\ln(1+\sqrt{2})\right).
	\end{align}
\end{ex}
\par This last result sparked our curiosity and prompted us to investigate the generalized series that will be discussed in a forthcoming paper. Another direction for future research is to examine infinite series covered in this work under the hypergeometric functions framework and tools.

\bigskip 
\hrule
\bigskip

\noindent{\it 2020 Mathematics Subject Classification:} 11B65, 05A10, 33B10, 40A05.

\noindent \emph{Keywords:} Infinite Series, Central Binomial Coefficients, Generating Functions, Pochhammer Symbol.

\bigskip
\hrule
\bigskip

\noindent (Concerned with sequence \seqnum{A000984}.)

	\bigskip
\hrule
\bigskip

\end{document}